\documentclass[11pt,draft]{amsart}

\newtheorem{theorem}{Theorem}
\newtheorem{lemma}[theorem]{Lemma}
\newtheorem{e-proposition}[theorem]{Proposition}
\newtheorem{corollary}[theorem]{Corollary}
\newtheorem{e-definition}[theorem]{Definition\rm}
\newtheorem{remark}{\it Remark\/}
\newtheorem{example}{\it Example\/}

\RequirePackage{amsmath}
\RequirePackage{amssymb}
\RequirePackage{amsthm}
 \RequirePackage{epsfig}

\begin{document}

\date{}

\title[On the Borel-Cantelli lemma and its generalization]
{On the Borel-Cantelli lemma and its generalization}

\author{Chunrong Feng, Liangpan Li, Jian Shen}

\address{Department of Mathematics, Shanghai Jiao Tong University,
Shanghai 200240, China $\&$ Department of Mathematical Sciences,
Loughborough University, Leics, LE11 3TU,  UK}
 \email{fcr@sjtu.edu.cn}

\address{Department of Mathematics, Shanghai Jiao Tong University,
Shanghai 200240,  China $\&$ Department of Mathematics, Texas State
University, San Marcos, TX 78666, USA }
 \email{liliangpan@yahoo.com.cn}

\address{Department of Mathematics, Texas State University, San Marcos,
TX 78666, USA} \email{js48@txstate.edu}



\date{}

\begin{abstract}
Let $\{A_n\}_{n=1}^{\infty}$ be a sequence of events on a
probability space $(\Omega,\mathcal{F},\mathbf{P})$. We show that if
$\lim_{m\rightarrow\infty}\sum_{n=1}^{m}w_n\mathbf{P}(A_n)=\infty$
where each $w_n\in\mathbb{R}$, then
\[{\mathbf{P}}(\limsup A_n)\geq\limsup_{n\rightarrow\infty}
\frac{\displaystyle\big(\sum_{k=1}^n{w_k\mathbf{P}}(A_k)\big)^2}{\displaystyle\sum_{i=1}^n\sum_{j=1}^nw_iw_j{\mathbf{P}}(A_i\cap
A_j)}.\]
\end{abstract}

\maketitle

\section{Introduction}
Let $\{A_n\}_{n=1}^{\infty}$ be a sequence of events on a
probability space $(\Omega,\mathcal{F},\mathbf{P})$. The classical
Borel-Cantelli lemma states that: (a) if
$\sum_{n=1}^{\infty}\mathbf{P}(A_n)<\infty$, then
$\mathbf{P}(\limsup A_n)=0$; (b) if
$\sum_{n=1}^{\infty}\mathbf{P}(A_n)$ $=\infty$ and
$\{A_n\}_{n=1}^{\infty}$ are mutually independent, then
$\mathbf{P}(\limsup A_n)$ $=1$. Here $\limsup
A_n=\bigcap_{n=1}^{\infty}\bigcup_{k=n}^{\infty}A_k$.
 The
Borel-Cantelli lemma played an exceptionally important role in
probability theory. Many investigations were devoted to the second
part of the Borel-Cantelli lemma in the attempt to weaken the
independence condition on $\{A_n\}_{n=1}^{\infty}$.

Erd\"{o}s and R\'{e}nyi \cite{ErdosRenyi} proved that the mutual
independence condition on $\{A_n\}_{n=1}^{\infty}$ can be replaced
by the weaker condition of pairwise independence. Indeed they
\cite{Renyi} (see also \cite{Dawson,KochenStone,Spitzer}) proved a
more general theorem:  if
$\sum_{n=1}^{\infty}\mathbf{P}(A_n)=\infty$, then
\[\mathbf{P}(\limsup A_n)\geq\limsup_{n\rightarrow\infty}
\frac{\displaystyle\big(\sum_{k=1}^n{\mathbf{P}}(A_k)\big)^2}{\displaystyle\sum_{i=1}^n\sum_{j=1}^n{\mathbf{P}}(A_i\cap
A_j)}.\]

There are many discussions and generalizations towards the
Borel-Cantelli lemma, for example see
\cite{Amghibech,Mori,Ortega,Yan}. The main purpose of this paper is
to present a weighted version of
the Erd\"{o}s-R\'{e}nyi  theorem:\\

\begin{theorem} \label{mainresult}
Suppose
$\lim_{m\rightarrow\infty}\sum_{n=1}^{m}w_n\mathbf{P}(A_n)=\infty$,
where each $w_n$ is a real weight (which could be negative). Then
\[{\mathbf{P}}(\limsup A_n)\geq\limsup_{n\rightarrow\infty}
\frac{\displaystyle\big(\sum_{k=1}^n{w_k\mathbf{P}}(A_k)\big)^2}{\displaystyle\sum_{i=1}^n\sum_{j=1}^nw_iw_j{\mathbf{P}}(A_i\cap
A_j)}.\]\\
\end{theorem}

The proof of Theorem \ref{mainresult} is relatively easy if we
further assume  all terms of the weight sequence to be nonnegative.
By choosing each $w_n = 1/{\mathbf{P}}(A_n)$ in
Theorem~\ref{mainresult}, we obtain the following corollary:\\

\begin{corollary}\label{interestedcorollary}
Suppose  ${\mathbf{P}}(A_n)>0$ holds for all $ n\in\mathbb{N}$. Then
\[{\mathbf{P}}(\limsup
A_n)\geq\limsup_{n\rightarrow\infty} \frac{\displaystyle n^2}
{\displaystyle\sum_{i=1}^n\sum_{j=1}^n{\frac{\displaystyle
\mathbf{P}(A_i\cap A_j)}{\displaystyle
\mathbf{P}(A_i)\mathbf{P}(A_j)}}}.\]\\
\end{corollary}


\section{Proof of the main result}

For any matrix $E=(z_{ij})_{m\times n}$,  denote by $\Gamma(E)$ the
sum of all its entries, that is,
$\Gamma(E)=\sum_{i=1}^m\sum_{j=1}^nz_{ij}$.

\begin{lemma}\label{cauchy}
Given a partition of an $(m+n)\times(m+n)$ symmetric matrix $E$:
$$E=\left ( \begin{array}{cc} A_{m\times m} & C_{m\times n} \\
C_{m\times n}^T & B_{n\times n}
\end{array} \right ).$$
If $E$ is positive semi-definite, then
$\Gamma(C)^2\leq\Gamma(A)\Gamma(B).$\\
\end{lemma}


\noindent {\it Proof: } This lemma follows from the following
inequality: $\forall x,y\in\mathbb{R}$,
\[(x,\ldots,x,y\ldots,y)E(x,\ldots,x,y\ldots,y)^{T}=\Gamma(A)x^2+2\Gamma(C)xy+\Gamma(B)y^2\geq0.\]
\hfill $\Box$\\

\begin{lemma}\label{semiprob} Let $\{A_i\}_{i=1}^{n}$ be finitely many events on a
probability space $(\Omega,\mathcal{F},\mathbf{P})$. Then the matrix
$\big(\mathbf{P}(A_i\cap A_j)\big)_{n\times n}$ is positive
semi-definite.\\
\end{lemma}

\noindent {\it Proof: }Let $\mathbb{E}(\cdot)$ be the expectation
function and $\chi_{A_i}$ be the indicator function of the event
$A_i$. Then $\mathbf{P}(A_i)=\mathbb{E}(\chi_{A_i})$ and
$\mathbf{P}(A_i\cap A_j)=\mathbb{E}(\chi_{A_i}\chi_{A_j})$. For each
$(s_1,s_2,\ldots,s_n)$ $\in\mathbb{R}^n$,
\begin{align*}
(s_1,s_2,\ldots,s_n)\big(\mathbf{P}(A_i\cap
A_j)\big)(s_1,s_2,\ldots,s_n)^{T}&=\sum_{i=1}^n\sum_{j=1}^ns_is_j\mathbf{P}(A_i\cap
A_j)
=\sum_{i=1}^n\sum_{j=1}^ns_is_j\mathbb{E}(\chi_{A_i}\chi_{A_j})\\
&=\mathbb{E}\big(\sum_{i=1}^n\sum_{j=1}^ns_is_j\chi_{A_i}\chi_{A_j}\big)=\mathbb{E}\big((\sum_{i=1}^ns_i\chi_{A_i})^2\big)\geq0.
\end{align*} This proves the lemma.
\hfill $\Box$\\

\begin{lemma}\label{same} Suppose
$\lim_{m\rightarrow\infty}\sum_{n=1}^{m}w_n{\mathbf{P}}(A_n)=\infty$,
where each $w_n \in\mathbb{R}.$
 Then
\[\lim_{n\rightarrow\infty}
\frac{\displaystyle\sum_{i=1}^n\sum_{j=1}^nw_iw_j{\mathbf{P}}(A_i\cap
A_j)}
{\displaystyle\sum_{i=2}^n\sum_{j=2}^nw_iw_j{\mathbf{P}}(A_i\cap
A_j)}=1.\]\\
\end{lemma}

\noindent {\it Proof: }By Lemma \ref{semiprob},
$E_n\doteq\Big(w_iw_j\mathbf{P}(A_i\cap A_j)\Big)_{n\times n}$ is
positive semi-definite. Define
\[
A=\Big(w_1w_1\mathbf{P}(A_1\cap A_1)\Big),\ \
B_n=\Big(w_iw_j\mathbf{P}(A_i\cap A_j)\Big)_{2\leq i, j\leq n},\ \
C_n=\Big(w_1w_j\mathbf{P}(A_1\cap A_j)\Big)_{2\leq j\leq n}.
\]
By Lemma \ref{cauchy},  $\Gamma(C_n)^2\leq \Gamma(A)\Gamma(B_n)\ \
(\forall n\geq2)$. By the Cauchy-Schwarz inequality,
\begin{align*}
\label{comparison}\big(\sum_{i=2}^nw_i\mathbf{P}(A_i)\big)^2 &
=\Big(\mathbb{E}\big(\sum_{i=2}^nw_i\chi_{A_i}\big)\Big)^2 \leq
\mathbf{P}(\bigcup_{i=2}^nA_i)\cdot
\mathbb{E}\Big(\big(\sum_{i=2}^nw_i\chi_{A_i}\big)^2\Big) \\ & =
\mathbf{P}(\bigcup_{i=2}^nA_i)\cdot
\big(\sum_{i=2}^n\sum_{j=2}^n{w_iw_j{\mathbf{P}}(A_i\cap A_j)}\big)
\leq \Gamma(B_n).
\end{align*}
Since
$\lim_{n\rightarrow\infty}\sum_{i=2}^{n}w_i{\mathbf{P}}(A_i)=\infty$,
we have $\Gamma(B_n)\rightarrow\infty$ as $n$ approaches to
infinity. Hence
\[\lim_{n\rightarrow\infty}\frac{\Gamma(A)+\Gamma(B_n)+2\Gamma(C_n)}{\Gamma(B_n)}=1+
\lim_{n\rightarrow\infty}\frac{2\Gamma(C_n)}{\Gamma(B_n)}=1.\]
This proves the lemma. \hfill $\Box$\\

\begin{remark} We obtained the following by-product from the proof of the above lemma:
\begin{equation} \label{weighted Chung-Erdos inequality}
\big(\sum_{i=1}^nw_i\mathbf{P}(A_i)\big)^2\leq\mathbf{P}(\bigcup_{i=1}^nA_i)\cdot
\big(\sum_{i=1}^n\sum_{j=1}^n{w_iw_j{\mathbf{P}}(A_i\cap A_j)}\big).
\end{equation}
This formula can be viewed as a weighted version of the
Chung-Erd\"{o}s inequality (\cite{ChungErdos}).\\
\end{remark}

\begin{e-proposition}\label{prop} Suppose
$\lim_{m\rightarrow\infty}\sum_{n=1}^{m}w_n{\mathbf{P}}(A_n)=\infty$,
where each $w_n \in\mathbb{R}.$ Then for all $ s\in\mathbb{N}$,
\[\limsup_{n\rightarrow\infty}
\frac{\displaystyle\big(\sum_{k=1}^nw_k{\mathbf{P}}(A_k)\big)^2}{\displaystyle\sum_{i=1}^n\sum_{j=1}^n{w_iw_j{\mathbf{P}}(A_i\cap
A_j)}}= \limsup_{n\rightarrow\infty}
\frac{\displaystyle\big(\sum_{k=s}^nw_k{\mathbf{P}}(A_k)\big)^2}{\displaystyle\sum_{i=s}^n\sum_{j=s}^n{w_iw_j{\mathbf{P}}(A_i\cap
A_j)}}.\]\\
\end{e-proposition}

Proposition \ref{prop} is an immediate corollary of Lemma
\ref{same}. With all the above preparation in hand, we are ready
to prove Theorem \ref{mainresult}.\\

\noindent  \emph{Proof of Theorem \ref{mainresult}}: By
(\ref{weighted Chung-Erdos inequality}) and Proposition \ref{prop},
\begin{align*}
{\mathbf{P}}(\limsup
A_n)&=\mathbf{P}(\bigcap_{s=1}^{\infty}\bigcup_{k=s}^{\infty}A_k)
=\lim_{s\rightarrow\infty}{\mathbf{P}}(\bigcup_{k=s}^{\infty}A_k)
=\lim_{s\rightarrow\infty}\Big(\lim_{n\rightarrow\infty}{\mathbf{P}}(\bigcup_{k=s}^{n}A_k)\Big)\\
&\geq
\lim_{s\rightarrow\infty}\Big(\limsup_{n\rightarrow\infty}\frac {
\left ( \displaystyle \sum _{k=s}^n w_k{\mathbf{P}}(A_k) \right ) ^2
} {\displaystyle \sum _{i=s}^n\sum _{j=s}^nw_iw_j
{\mathbf{P}}(A_i\cap A_j)}\Big) =\limsup_{n\rightarrow\infty}\frac {
\left ( \displaystyle \sum _{k=1}^n w_k{\mathbf{P}}(A_k) \right ) ^2
} {\displaystyle \sum _{i=1}^n\sum _{j=1}^n
w_iw_j{\mathbf{P}}(A_i\cap A_j)}.
\end{align*}
This completes the proof of Theorem \ref{mainresult}.\\

\begin{corollary}
Let $\{w_n\geq0\}_{n=1}^{\infty}$ be a bounded sequence with
$\sum_{n=1}^{\infty}w_n\mathbf{P}(A_n)=\infty$. Then
\[{\mathbf{P}}(\limsup A_n)\geq\limsup_{n\rightarrow\infty}
\frac{\displaystyle\sum_{1\leq i<j\leq
n}w_iw_j\mathbf{P}(A_i)\mathbf{P}(A_j)}{\displaystyle\sum_{1\leq
i<j\leq n}w_iw_j\mathbf{P}(A_i\cap A_j)}.\]
\end{corollary}

\begin{proof}
By the weighted version of the Chung-Erd\"{o}s inequality
(\ref{weighted Chung-Erdos inequality}),
\[\sum_{i=1}^n\sum_{j=1}^n{w_iw_j{\mathbf{P}}(A_i\cap A_j)}\geq\big(\sum_{i=1}^nw_i\mathbf{P}(A_i)\big)^2\ \ \ \ (\forall n\in\mathbb{N}),\]
which yields
\begin{equation}\label{A1}\lim_{n\rightarrow\infty}\frac{\displaystyle\sum_{i=1}^nw_iw_i{\mathbf{P}}(A_i\cap A_i)}{
\displaystyle\sum_{i=1}^n\sum_{j=1}^n{w_iw_j{\mathbf{P}}(A_i\cap
A_j)}}=0\end{equation} by considering $\{w_n\geq0\}_{n=1}^{\infty}$
is a bounded sequence with
$\sum_{n=1}^{\infty}w_n\mathbf{P}(A_n)=\infty$. Note also
\begin{equation}\label{A2}
\frac{\displaystyle\big(\sum_{k=1}^n{w_k\mathbf{P}}(A_k)\big)^2}{\displaystyle\sum_{i=1}^n\sum_{j=1}^nw_iw_j{\mathbf{P}}(A_i\cap
A_j)} \geq \frac{\displaystyle2\cdot\sum_{1\leq i<j\leq
n}w_iw_j\mathbf{P}(A_i)\mathbf{P}(A_j)}{\displaystyle\sum_{i=1}^n\sum_{j=1}^n{w_iw_j{\mathbf{P}}(A_i\cap
A_j)}}.\end{equation} Combining (\ref{A1}), (\ref{A2}) and Theorem
\ref{mainresult} yields the desired result.
\end{proof}



\begin{example}Let $(\Omega,\mathcal{F},\mathbf{P})$ be a probability space,
$A,B\in\mathcal{F}$, $\mathbf{P}(A\cap B)>0$. For all
$n\in\mathbb{N}$, let $A_{3n-2}=A$, $A_{3n-1}=B$, $A_{3n}=A\cap B$.
By the Erd\"{o}s-R\'{e}nyi  theorem,
\[{\mathbf{P}}(\limsup
A_n)\geq\frac{\big(\mathbf{P}(A)+\mathbf{P}(B)+\mathbf{P}(A\cap
B)\big)^2}{\mathbf{P}(A)+\mathbf{P}(B)+7\mathbf{P}(A\cap B)}.\]
Applying Theorem \ref{mainresult} with the weight sequence
$1,1,-1,1,1,-1,1,1,-1,\ldots$, we obtain
\[{\mathbf{P}}(\limsup
A_n)\geq \mathbf{P}(A)+\mathbf{P}(B)-\mathbf{P}(A\cap
B)=\mathbf{P}(A\cup B).\] In fact ${\mathbf{P}}(\limsup
A_n)=\mathbf{P}(A\cup B)$. So Theorem~\ref{mainresult} is best
possible for this example.
\end{example}




\section*{Acknowledgements}
Feng would like to acknowledge
the support of National Basic Research Program of China (973
Program No. 2007CB814903), National Natural Science Foundation of
China (No. 70671069) and the Mathematical Tianyuan Foundation of
China (No. 10826090). Li's research was partially supported by the
Mathematical Tianyuan Foundation of China (No. 10826088). Shen's
research was partially supported by NSF (CNS 0835834) and Texas
Higher Education Coordinating Board (ARP 003615-0039-2007).

\end{document}